\documentclass[a4paper, 11pt, reqno]{amsart}
\usepackage{amsmath,amsthm,amssymb,amscd}
\usepackage[english]{babel}
\usepackage[ansinew]{inputenc}
\usepackage[T1]{fontenc}

\usepackage[all]{xy}
\usepackage{xspace}
\usepackage{mathrsfs}
\usepackage[dvips]{color}
\usepackage{epsfig}
\usepackage{float}
\usepackage{wasysym}
\usepackage{setspace}
\usepackage{overpic}
\usepackage{enumerate}

\usepackage{hyperref}

\def\g{\gamma}
\def\G{\Gamma}
\def\si{\sigma}

\def\vp{\varphi}
\def\ep{\varepsilon}
\def\la{\lambda}

\def\de{\delta}
\def\oo{\mbox{$\Omega$} }

\newcommand{\VV}{{\mathcal V}}
\newcommand{\OO}{{\mathcal O}}

\newcommand{\CC}{{\mathcal C}}

\newcommand{\N}{{\mathbb N}}

\newcommand{\R}{{\mathbb R}}
\newcommand{\C}{{\mathbb C}}

\newcommand{\D}{{\mathbb D}}

\newcommand{\cbar}{{\widehat{\mathbb C}}}

\newcommand{\diam}{\mathrm{diam}}
\newcommand{\mesh}{\operatorname{mesh}}
\newcommand{\emty}{\emptyset}

\newcommand{\crit}{\hbox{\rm crit}}
\newcommand{\post}{\hbox{\rm post}}

\newcommand{\id}{\mathop{\rm id}}

\newcommand{\ben}{\begin{enumerate}}
\newcommand{\een}{\end{enumerate}}

\newcommand{\X} {\mathbf{X}}

\providecommand{\abs}[1]{\lvert#1\rvert}

\providecommand{\norm}[1]{\lVert#1\rVert}

\numberwithin{figure}{section}

\newtheorem{theorem}{Theorem}[section]
\newtheorem{lemma}[theorem]{Lemma}

\newtheorem{proposition}[theorem]{Proposition}
\newtheorem{corollary}[theorem]{Corollary}
\newtheorem{definition}[theorem]{Definition}

\newenvironment{remark}[1][Remark.]{\begin{trivlist}
\item[\hskip \labelsep \textsc{#1}]}{\end{trivlist}}

\theoremstyle{remark}
\newtheorem{ex}[theorem]{Example}
\newtheorem*{claim*}{Claim}
\newtheorem{rem}[theorem]{Remark}
\makeatletter
\renewcommand\theequation{\thesection.\arabic{equation}}
\@addtoreset{equation}{section}
\makeatother

\makeatletter
\newcommand{\eqnum}{\leavevmode\hfill\refstepcounter{equation}\textup{\tagform@{\theequation}}}
\makeatother

\title{Invariant Jordan curves of Sierpi\'{n}ski carpet rational
maps}
\author{Yan Gao}
\address{Yan Gao, Mathemaitcal School of Sichuan University, 610065, P.R.China.}
\email{gyan@scu.edu.cn}
\author{Peter Ha\"{\i}ssinsky}
\address{Peter Ha\"{\i}ssinsky,
Universit\'e d'Aix-Marseille,
Institut de Math\'ematiques de Marseille (I2M),
39, rue Fr\'ed\'eric Joliot Curie
13453 Marseille Cedex 13
France.}
\email{phaissin@math.univ-toulouse.fr}
\author{Daniel Meyer}
\address{Daniel Meyer, Department of Mathematics and Statistics, P.O.Box 35,
  FI-40014 University of Jyv\"{a}skyl\"{a}, Finland.}
\email{dmeyermail@gmail.com}
\author{Jinsong Zeng}
\address{Jinsong Zeng, Academy of Mathematics and Systems Science, Chinese Academy of Sciences, Beijing 100190 P.R. China.}
\email{zeng.jinsong@amss.ac.cn}

\date{\today}

\begin{document}

\maketitle
\begin{abstract}
 In this paper, we prove that if $R:\cbar\to\cbar$ is a postcritically finite rational map
  with Julia set ho\-meo\-mor\-phic to the Sierpi\'nski
  carpet, then
  there is an integer $n_0$, such that, for any $n\ge n_0$,
  there exists an $R^n$-invariant Jordan curve $\Gamma$
  containing the postcritical set of $R$.

\end{abstract}
\begin{quote}
\end{quote}
\tableofcontents

\section{Introduction}
A central problem in Complex Dynamics is the classification of rational maps of a given degree up to M\"obius
conjugation. A general approach to this problem is to structure this space of maps around the so-called
postcritically finite rational maps. In the case of postcritically finite polynomials, Douady and Hubbard have introduced the
so-called Hubbard trees which capture their dynamical features \cite{DH2}.
A long-standing problem is to develop analogous combinatorial
invariants for general rational maps, cf. \cite[Problem 5.5]{ctm:classification}. The first works on this problem
concerned degree two rational maps \cite{wittner:thesis, bernard:thesis, rees:views}. For postcritically finite Newton methods, there has also been
recent progress  \cite{lms1,lms2}. In general degree, combinatorial tools have been introduced leading to the  construction of
rational maps from subdivision rules \cite{cfkp:fsrrat}; see also \cite{cgnpp} where the authors provide a conjectural picture
of critically fixed rational maps from finite graphs, and \cite{rosetti:thesis} where the author looks for planar graphs characterizing 
postcritically finite rational maps (the present work is used in the latter).


In \cite{bonk:meyer:expanding,  cfp:ratofsr}, the authors prove independently

\begin{theorem}
  \label{thm:R_exp_inv_C}
  Let $R:\cbar\to\cbar$ be a postcritically finite rational map
  with no critical periodic cycle. Then there is an integer
  $n_0$, such that, for any $n\ge n_0$, there exists an
  $R^n$-invariant Jordan curve $\Gamma$ containing the
  postcritical set of $R$.
\end{theorem}

Here by $\Gamma$ is \emph{$R^n$-invariant} we mean that it is
\emph{$R^n$-forward invariant}, i.e., $R^n(\Gamma) \subset
\Gamma$.
The curve $\Gamma$ enables us to introduce a Markov partition and to develop a combinatorial description of the dynamics.
Actually, this theorem is proved in the broader setting of {\it expanding Thurston maps}
in \cite{bonk:meyer:expanding}, cf. Theorem \ref{thm:inv_C_expT}. Building on this work, we shall
prove in the present paper

\begin{theorem}
  \label{thm:main}
  Let $R:\cbar\to\cbar$ be a postcritically finite rational map
  with Julia set ho\-meo\-mor\-phic to the Sierpi\'nski
  carpet. Then
  there is an integer $n_0$, such that, for any $n\ge n_0$,
  there exists an $R^n$-invariant Jordan curve $\Gamma$
  containing the postcritical set of $R$.
\end{theorem}

The proof actually shows that we may choose $\Gamma$ in any
homotopy class relative to the postcritical points. More
precisely for any Jordan curve $\widetilde{\Gamma}\subset \cbar$
with $\post(f) \subset \widetilde{\Gamma}$ there is $n_0 =
n_0(\widetilde{\Gamma})$ such that for all $n\geq n_0$ there is
an $R^n$-invariant Jordan
curve $\Gamma\subset \cbar$ that
is isotopic to $\widetilde{\Gamma}$ relative to the postcritical
points (thus contains all postcritical points).

\begin{remark} In \cite{rees:invariantgraph}, Rees constructs a Markov partition for geometrically finite rational maps
with Julia sets homeomorphic to the Sierpi\'nski carpet. Her approach is very different.\end{remark}

\medskip

There are many examples of  postcritically finite rational maps with Julia sets homeomorphic to the Sierpi\'nski carpet.
The first example is due to Milnor and Tan Lei \cite[Appendix]{milnor:tanlei}; see also
the survey \cite{devaney:bams13} and the references therein for many others. These rational maps and their Julia sets
play a particular role. These maps are centers of  hyperbolic components. Conjecturally, those components are relatively compact
in the space of rational functions up to M\"obius conjugation \cite[Question 5.3]{ctm:classification} (established in degree $2$ in \cite{aepstein:etds}). 
Moreover, their Julia sets  turn out to be rigid in a very strong sense \cite{bonk:lyubich:merenkov}.

\subsection{Outline}
\label{sec:outline}

The proof of Theorem~\ref{thm:main} proceeds as follows. Let
$R\colon \cbar \to \cbar$ be a postcritically finite rational
map with Sierpi\'{n}ski carpet Julia set. Collapse the closure
of each Fatou component. This yields an expanding Thurston map
$f\colon S^2\to S^2$ in the sense of
\cite{bonk:meyer:expanding} by Theorem \ref{thm:quotient}. Using
Theorem~\ref{thm:inv_C_expT}, or rather the appropriate
version
Theorem~\ref{thm:invC_subdivSector}, we obtain a Jordan curve $\CC$
that is invariant for any sufficiently high iterate $f^n$ and whose lift  to $\cbar$ contains an $R^n$-invariant Jordan
curve, see Theorem~\ref{thm:lifts} and its proof.

\subsection{Notation}
\label{sec:notation}

The $2$-sphere is denoted by $S^2$, the Riemann sphere by
$\cbar$ and the unit disk by $\D$. The set of critical points of a branched covering map
$f$ is denoted by $\crit(f)$, the set of postcritical points by
$\post(f)$ (see Section~\ref{sec:branched-coverings}). The Julia
set of a rational map $R$ will be denoted by $\mathcal{J}_R$; the Fatou
set is $\mathcal{F}_R$. Given two positive quantities $a$ and $b$, we will
write $a\lesssim b$ or $b\gtrsim a$
if there is a constant $C>0$
such that $a\le Cb$. The $\delta$-neighborhood of a set $A$ is
denoted by $\mathcal{N}^\delta(A)$.

\section{Branched covers, Thurston and rational maps}
\label{sec:thurston-maps}

Here we present some elementary background that will be used in
this paper. More details can be found in
\cite{bonk:meyer:expanding} and \cite{HaiPil}.

\subsection{Branched coverings}
\label{sec:branched-coverings}

A  map $f:S^2\to S^2$ is a  {\it branched covering} (of the
sphere $S^2$) if there are orientation-preserving homeomorphisms $\vp,\psi:S^2\to\cbar$ and
a rational map $R:\cbar\to\cbar$ such that $f= \psi^{-1}\circ R\circ\vp$. It is characterized
by being finite-to-one, open, and orientation-preserving cf.
\cite[Theorem~X.5.1]{whyburn:analytic_topology}.
The \emph{degree} $\deg f$ and
\emph{local degree at $x$} $\deg(f ,x)$ of $f$ are defined as the degree of $R$ and the local degree of $R$ at $\vp(x)$: $\deg (f, x)=d$
means there are charts in the neighborhoods of $x$ and $f(x)$ so that $f$ takes the form $z\mapsto z^d$.

The $n$-th iterate of $f$ will be denoted by $f^n$.
The  set of {\it critical points} $\crit(f)$ corresponds to the points  $c\in S^2$ such that  $\deg(f,c)>1$, i.e., $R'(\vp(c))=0$
(in a suitable chart).
The postcritical set $\post(f)$ is defined as
\begin{equation*}
  \post(f)=\{f^n(c) : {c\in\crit(f)}, {n\ge 1} \}.
\end{equation*}
Note that $f(\post(f))\subset \post(f)$ and $\post(f)=\post(f^n)$ for any $n\ge 1$.
A finite branched covering $f$ is {\it postcritically finite} if $\post(f)$ is a finite set.

If $f$ is postcritically finite, then every postcritical point
$p\in \post(f)$ is preperiodic, i.e., there are  minimal iterates
$k\geq 0$ and $m\geq 1$ such that $f^{k+m}(p)=f^k(p)$;
 $m$ is the \emph{period} of the
cycle $\{f^k(x),\ldots, f^{k+m-1}(x)\}$. If the  cycle does not contain a critical point,
we say that $p$ is of \emph{Julia-type}, and we denote by $ \post_{\mathcal{J}}(f)$
the set of Julia-type postcritical points. Otherwise,
 the periodic cycle
contains a critical point  and we say $p$ is of
\emph{Fatou-type}; let   $\post_{\mathcal{F}}(f)$ denote the set
of postcritical points of Fatou-type.
Note that $p\in \post_{\mathcal{F}}(f)$ if and only if
$\deg(f^n,p)\to \infty$ as $n\to \infty$. 
The terminology is of course explained by the
well-known fact that if  $f$ is a rational map then a
postcritical point $p$ is of Julia-type if and only if $p$ is
contained in the Julia set and is of Fatou-type if and only if
$p$ is contained in the Fatou set of $f$.

We mention the following fact for future reference, which is a consequence of the definitions
and of the Riemann-Hurwitz formula.

\begin{lemma}
  \label{lem:pre_1post}
  Let $f\colon S^2\to S^2$ be a finite branched covering and $V\subset S^2$
  be a Jordan domain that contains at most a single
critical value $p$. Then every component $U$ of
  $f^{-1}(V)$ is a Jordan domain and $f:U\to V$ is equivalent to $z^d$, where $d= \deg(f,c)$ and $c$ is the unique
  preimage of $p$ in $U$, i.e., there are orientation-preserving homeomorphisms $\varphi:U\to \D$ and $\psi:V\to\D$
such that $\varphi(c)=\psi(f(c))=0$ and $\psi\circ f\circ\varphi^{-1}:\D\to\D$ is
  $z\mapsto z^d$.
\end{lemma}



\subsection{Thurston maps}
\label{sec:thurston-maps-1}

The branched covers of the sphere that we will be considering
will not always be rational. In this situation, the following
notion is very useful.

\begin{definition}\label{def:f}
  A \emph{Thurston map} is an orientation-preserving, postcritically
  finite, branched covering of the sphere,
  \begin{equation*}
    f\colon S^2\to S^2.
  \end{equation*}
We fix a base metric $\sigma$ on $S^2$
that induces the standard
topology on $S^2$. Consider a Jordan curve $\CC\supset \post f$. The Thurston map
  $f$ is
  called \emph{expanding} if
    \begin{equation*}   \label{def:fexpanding}
      \mesh f^{-n}(\CC) \to 0 \text{ as } n\to \infty.
    \end{equation*}
where  $\mesh f^{-n}(\CC)$ denotes the maximal diameter of a
  component of $S^2\setminus f^{-n}(\CC)$.
\end{definition}
 In
  \cite[Lemma~6.1]{bonk:meyer:expanding}{} it
  was shown that this definition is independent of the chosen curve
  $\CC$. This notion of ``expansion'' agrees with the one by
  Ha\"{i}ssinsky-Pilgrim in \cite{HaiPil} (see
  \cite[Proposition~6.3]{bonk:meyer:expanding}).



\medskip

Let $f:S^2\to S^2$ be a Thurston map
and fix a Jordan curve $\CC\subset S^2$ with $\post(f) \subset \CC$.
The closure of one of the two components of $S^2\setminus \CC$
is called a \emph{$0$-tile}. Similarly, we call the closure
of one component of $S^2\setminus f^{-n}(\CC)$ an
$n$\emph{-tile} (for any $n\in \N_0$). The set of all $n$-tiles
is denoted by
$\X^n(\CC)$. For any $n$-tile $X$, the set $f^n(X)=X^0$ is a $0$-tile
and
\begin{equation}
  \label{eq:fnXntoXhomeo}
  f^n\colon X\to X^0\quad \text{is a homeomorphism,}
\end{equation}
see \cite[Proposition~5.17]{bonk:meyer:expanding}.  This means
in particular that each $n$-tile is a closed Jordan domain. The
definition of ``expansion'' implies that $n$-tiles become
arbitrarily small, this is the (only) reason we require
expansion.

We call the points in $f^{-n}(\post(f))$ the
\emph{$n$-vertices}, so the postcritical points are exactly the $0$-vertices. The closure of any component of
$f^{-n}(\CC)\setminus f^{-n}(\post(f))$ is an
\emph{$n$-edge}. Thus the $0$-edges are precisely the closed
arcs into which the points $\post(f) \subset \CC$ divide $\CC$.

\smallskip
The $n$-tiles, $n$-edges, $n$-vertices form a \emph{cell complex} when
viewed as $2$-, $1$-, and $0$-cells (see
\cite[Chapter~5]{bonk:meyer:expanding}).

If we consider just the $n$-edges and $n$-vertices, we obtain a
$1$-dimensional cell complex, i.e., a \emph{graph} in the natural
way. Note that this graph may have multiple edges, but no loop
edges.

An
\emph{$n$-edge path} is a path in this graph, meaning that it is
a finite sequence $e_1, \dots, e_N$ of $n$-edges, where $e_j\cap
e_{j+1}$ is an $n$-vertex
for all $j=1,\dots,N-1$.
Such an $n$-edge path is \emph{simple} if furthermore $e_i \cap
e_j= \emptyset$ for $\abs{i-j}\geq 2$. We allow an
$n$-path to consist of only a single $n$-vertex, which is then
simple.

\subsection{Fatou dynamics}
\label{bottcher}
Let us now assume again that $R\colon \cbar \to \cbar$ is a
postcritically finite rational map. Then each Fatou component
$\Omega$ of $R$ is simply connected, cf.  \cite[Remark, p.\,35]{McM2}).
From \emph{Boettcher's theorem} it then follows that there is a conformal map $\eta_{\Omega}:\D\to \oo$ and some power $d_{\Omega}$
such that $R\circ \eta_{\Omega}(z)=
\eta_{R(\Omega)}(z^{d_{\Omega}})$ for all $z\in\D$.

Furthermore, the Julia set $\mathcal{J}_R$ of $R$ is locally
connected, see \cite[Theorem~19.7]{milnor:dynamics}.
Thus it follows from Carath\'eodory's theorem that the
conformal map $\eta$ extends to a continuous and surjective map
$\eta_{\Omega}:\overline{\D}\to\overline{\Omega}$.

An {\it internal ray} is the image $\eta_{\Omega}([0,1)\si)$ for
some Fatou component $\Omega$ and some complex number of modulus one $\si\in\partial\D$.
Note that internal rays are mapped to internal rays under $R$.

\subsection{Orbifold metric}
\label{sec:orbifold-metric}
Again let $R\colon \cbar \to \cbar$ be a postcritically finite
rational map (or rational Thurston map).  For any such map $R$,
there exists  a complete metric $d_{\OO}$
 called the {\it  orbifold metric} on $\OO= \cbar\setminus \post_{\mathcal{F}}(R)$,  the complement of the postcritical points
of Fatou-type, with
the following properties (see \cite[Appendix~E]{milnor:dynamics} as well as \cite[Appendix~A.9]{bonk:meyer:expanding}). It is induced by a conformal  metric  $\rho(z)|dz|$ with $\rho$ smooth in the complement
of $\post(R)$, and, for any $w\in \OO$ and $z\in R^{-1}(w)$,
it satisfies $\norm{R'(z)}_{\OO}>1$. Thus for any compact set
$K\subset \OO$ there is a constant $\la>1$ such that
$\|R'(z)\|_{\OO}\ge \la$ for all $z\in R^{-1}(K)$. See
\cite[Theorem 19.6]{milnor:dynamics} for details.

From the previous section it is easy to see that we may choose
such a relatively compact set $W\subset \OO$ such that $W'=R^{-1}(W)$ has  compact closure
in $W$ and such that $W$ contains all Fatou components
that do not contain a Fatou-type postcritical point. Let us then
fix the constant $\lambda>1$ for this set $W$ as above.

\begin{lemma}[Shrinking Lemma]
  \label{lem:shrink}
  Let $R\colon \cbar \to \cbar$ be a postcritically finite
  rational map. There are constants $\delta>0$ and $C>0$
  such that the following holds. Let  
   $K\subset \OO$ be a compact set intersecting $\mathcal{J}_R$ 
  and diameter at most $\de$, then for any iterate $n\ge 1$ and
  any component $L$ of $R^{-n}(K)$, 
  the inequality 
  \begin{equation*}
    \diam_\OO (L)\le C \diam_\OO (K)/\la^n
  \end{equation*}
  holds. 

\end{lemma}

\begin{proof}
Since $d_\OO$ comes from a conformal metric, we may find $\chi\ge 1$ and $r_0>0$ such that, for any $z\in W$, any $r\in (0,r_0]$, we can
find a Jordan domain $D$ such that $B_\OO(z,r)\subset D\subset B_\OO(z,\chi r)$. Let $\de\in (0,r_0]$ be small enough so that,
for any distinct postcritical points $z,w$ of Julia type, $d_\OO(z,w)> \chi^2\de$ and $d_\OO(\mathcal{J}_R,\OO\setminus W)> \chi^2 \de$.
If $K$ is a compact subset with connected complement  of diameter at most $\de$, then either (a) we can find
$z\in K$ and a simply connected domain $D$ such that  $B_\OO(z,\diam_\OO K)\subset D\subset B_\OO(z,\chi \diam_\OO K)$ and
$B_\OO(z,\chi \diam_\OO K)$ is disjoint from $\post(R)$, or (b) we can
find $z\in \post_{\mathcal{J}}(R)$ and a simply connected set $D$ such that $K\subset B_{\OO}(z, \chi\diam_\OO K)\subset D\subset  B_{\OO}(z, \chi^2\diam_\OO K)$ and
$ B_{\OO}(z, \chi\diam_\OO K)\cap \post (R)=\{z\}$.
Fix $n\ge 1$; according to Lemma \ref{lem:pre_1post}, for  any component $D'$ of $R^{-n}(D)$, $R^{-n}(\{z\})\cap D'$ is a singleton $\{z'\}$
and, for any $x\in D'$, $d_\OO(x,z')\le d_\OO(R^n(x),z)/\la^n$.  Therefore, if $L$ is a component of $R^{-n}(K)$, then
$\diam_\OO L\le2  \diam_\OO K/\la^n$ holds in case (a),  and, in case (b), we obtain $\diam_\OO L\le 2\chi \diam_\OO K/\la^n$.
\end{proof}

\subsection{Sierpi\'{n}ski carpet Julia sets}
\label{sec:sierp-carp-julia}

Let us now assume additionally that the Julia set $\mathcal{J}_R$ of $R$ is
a Sierpi\'{n}ski carpet. This means by definition that $\mathcal{J}_R$ is
homeomorphic to the standard Sierpi\'{n}ski carpet. By
\emph{Whyburn's characterization}, a set $S\subset \cbar$ is a
Sierpi\'nski carpet if and only if it is compact, connected,
locally connected, has no local cut-points, and has topological
dimension $1$, see \cite{whyburn:sierpinski}. This implies that
each component of $\cbar\setminus S$ is a Jordan domain,
distinct such complementary components have disjoint closures.
Thus our Julia set $\mathcal{J}_R$ has these properties.

This means that each component of the Fatou set is a Jordan
domain and that distinct components of the Fatou set have
disjoint 
closures. Furthermore, the boundary of a component of the  Fatou
set cannot contain a critical value $v$, since any critical
point $c$ with $R(c)=v$ would be a local cutpoint.
The  local connectedness of a continuum implies that it is an \emph{$E$-continuum}\label{Econt}
 (for any $\ep>0$, only finitely many Fatou components have
diameter larger than $\ep$); see \cite[Theorem
 VI.4.4]{whyburn:analytic_topology}.  We record this fact as a lemma for future reference.

\begin{lemma}
  \label{lem:comp_F_small}
  Let $\epsilon>0$ be arbitrary. Then there are only finitely
  many components $\Omega$ of $\mathcal{F}_R$ with
  $\diam(\Omega)\geq \epsilon$.
\end{lemma}

\section{Review of invariant Jordan curves}
\label{sec:revi-invar-jord}

Recall from the introduction that in \cite{bonk:meyer:expanding}
the following result is given:
\begin{theorem}[{\cite[Theorem~14.1]{bonk:meyer:expanding}}]
  \label{thm:inv_C_expT}
  Let $f\colon S^2\to S^2$ be an expanding Thurston map. Then
  for each sufficiently large $n\in \N$ there is a Jordan curve
  $\CC\subset S^2$ with $\post(f)$ that is invariant for the
  iterate $f^n$ of $f$. 
\end{theorem}
Recall that this means that
  \begin{equation*}
    f^n(\CC) \subset \CC.
  \end{equation*}

 We will need a variant of this
theorem. More precisely, we will need to adjust the proof in one
specific place. Here an outline of the proof will be given, in
the next section we will prove a slightly stronger version of
this theorem. 

\smallskip
A combinatorial condition that will allow us to construct the
invariant curve is required.
\begin{definition}
  Let  $\textup{post}(f)\subset\CC\subset
  S^2$ be a Jordan curve. A set $K\subset S^2$
  \emph{joins opposite sides} of $\CC$ if
  \begin{itemize}
  \item $K$ intersects disjoint $0$-edges of $\CC$ in the case
    when $\#\post(f) \geq 4$;
  \item $K$ intersects all three $0$-edges of $\CC$ in the case
    when $\#\post(f) =3$.
  \end{itemize}
\end{definition}

Recall that the $0$-edges of $\CC$ are the closed arcs into
which the postcritical points subdivide $\CC$.

\begin{definition}
  Let $f\colon S^2\to S^2$ be a Thurston map and $\CC\subset
  S^2$ be an $f$-invariant Jordan curve with
  $\post(f)\subset\CC$. Then $f$ is \emph{combinatorially
    expanding} for $\CC$ if there is an $n\in \N$, such that no
  $n$-tile $X$ (defined in terms of $(f,\CC)$) joins opposite
  sides of $\CC$.
\end{definition}

If a Thurston map $f$ is expanding, the diameter of $n$-tiles
goes to $0$ as $n\to \infty$. Thus, an expanding Thurston map is
always combinatorially expanding for any $f$-invariant Jordan
curve $\CC$.

The following theorem
(which is \cite[Theorem~14.4 and
Corollary~14.14]{bonk:meyer:expanding})
is the main step in proving Theorem~\ref{thm:inv_C_expT}.

\begin{theorem}
  \label{thm:invC_fromCC1}
  Let $f\colon S^2\to S^2$ be an expanding Thurston map. Assume
  there exist Jordan curves $\CC, \CC'\subset S^2$ with
  $\post(f) \subset \CC,\CC'$ and $\CC'\subset f^{-1}(\CC)$, and
  an isotopy $H\colon S^2\times [0,1]\to S^2$ rel.\ $\post(f)$
  with $H_0=\id_{S^2}$ and $H_1(\CC)= \CC'$ such that the map
  \begin{equation*}
    \widehat{f}:= H_1 \circ f \text{ is combinatorially
      expanding for } \CC'.
  \end{equation*}
  Then there exist an $f$-invariant Jordan curve
  $\widetilde{\CC}\subset S^2$ with $\post(f) \subset
  \widetilde{\CC}$.

  Furthermore there is a homeomorphism $h\colon S^2\to S^2$
  satisfying the following:
  \begin{align*}
    &h(\post(f))= \post(f),
    \\
    &h(\widetilde{\CC}) = \CC',
    \\
    &\text{$h$ maps the $1$-tiles, $1$-edges, and
      $1$-vertices  for $(f,\widetilde{\CC})$ to the}
    \\
    &\text{\phantom{$\varphi$ maps the }$1$-tiles, $1$-edges,
      and $1$-vertices for $(f, \CC)$.}
  \end{align*}
\end{theorem}

The second statement basically says that the $1$-tiles for
$(f,\CC)$ divide the $0$-tiles for $(f,\CC')$ in the same
combinatorial fashion as the $1$-tiles for $(f,\widetilde{\CC})$
subdivide the $0$-tiles for $(f,\widetilde{\CC})$. Thus by
choosing the Jordan curves $\CC$ and $\CC'$ in a certain way, we
may ensure that the invariant curve $\widetilde{\CC}$ has
certain desired properties. This will be the theme of the next
section.


The following proposition and lemma finishes the proof of
Theorem~\ref{thm:inv_C_expT} in conjunction with
Theorem~\ref{thm:invC_fromCC1}. We will be using the following
notation. Let $\CC$ be an oriented Jordan curve and $p,q\in
\CC$. Then $\CC(p,q)$ denotes the closed arc on $\CC$ from $p$
to $q$.

\begin{proposition}
  [{\cite[Proposition~11.7 and proof of
    Lemma~11.17]{bonk:meyer:expanding}}]
  \label{prop:isotreln}
  Suppose $\CC$ is an oriented Jordan curve in $S^2$ and
  $P\subset \CC$ a
  set consisting of $n\ge 3$ distinct points
  $p_1, \dots, p_n, p_{n+1}=p_1$ in cyclic order on $\CC$.
  Then there exists
  $\delta>0$ and $\epsilon_0>0$ satisfying the following.

  \begin{enumerate}
  \item
    Let $\CC'$ be another Jordan curve in $S^2$ passing through
    the points of $P$ in the same cyclic order as $\CC$, and let
    $\CC'(p_i,p_{i+1})$ be the arc on $\CC'$ with endpoints
    $p_i$ and $p_{i+1}$. If
    \begin{equation*}
      \CC'(p_i,p_{i+1})
      \subset
      \mathcal{N}^\delta(\CC(p_i,p_{i+1}))
    \end{equation*}
    for all $i=1, \dots, n$, then there exists an isotopy $H_t$
    on $S^2$ rel.\ $P$ such that $H_0=\id_{S^2}$ and
    $H_1(\CC)=\CC'$.
  \item
    If $K\subset S^2$ is a set with $\diam(K)< \epsilon_0$ then
    $K$ does not join opposite sides of $\CC'$.
  \end{enumerate}
\end{proposition}

\begin{theorem}[{\cite[Lemma~11.17]{bonk:meyer:expanding}}]
  \label{thm:ex_C1_in_C}
  Let $f\colon S^2\to S^2$ be an expanding Thurston map. Let
  $\CC\subset S^2$ be an arbitrary Jordan curve with
  $\post(f)\subset \CC$ and $\delta>0$ be arbitrary. Then for
  each sufficiently large
  $n\in \N$ we can find a Jordan curve $\CC'\subset f^{-n}(\CC)$
  that traverses the the postcritical points of $f$ in the same
  cyclic order as $\CC$ and for any two postcritical points
  $p_i, p_{i+1}$ that are consecutive on $\CC$ we have
  \begin{equation*}
    \CC'(p_i,p_{i+}) \subset \mathcal{N}^\delta(\CC(p_i,p_{i+1})).
  \end{equation*}
\end{theorem}

\section{Invariant Jordan curves that subdivide sectors}
\label{sec:invar-jord-curv}

As we explained in the introduction, we will construct a quotient map $f$ from $R$ which will be expanding.
We wish to consider a forward invariant curve for (an iterate of) $f$ 
 and construct a forward-invariant curve within its lift. In
 general, this construction will not result in a Jordan curve
 going through the postcritical set of $R$, cf. Lemma
 \ref{lma:liftg}. 
We will be need a slight variant of
Theorem~\ref{thm:inv_C_expT}, which will be proved in this
section.

Let $f\colon S^2\to S^2$ be an expanding Thurston map. Assume
that the Jordan curve $\CC\subset S^2$ with $\post(f) \subset
\CC$ is $f$-invariant. Then the $n$-tiles \emph{subdivide} the
$0$-tiles. This means that each $n$-tile is contained in
exactly one $0$-tile. Let us fix a postcritical point $p$.
There are exactly two $0$-edges, $E, E'$ incident to $p$. Given
an $n\in \N$, let $e_1, \dots, e_N$ be the $n$-edges incident to
$p$, ordered cyclically around $p$. Note that $N= 2\deg(f^n,
p)$. Thus $N$, i.e., the number of such $n$-edges, goes to
$\infty$ as $n\to \infty$ if and only if $p$ is of
Fatou-type. Since $\CC$ is $f$-invariant, there are exactly two
of these $n$-edges, say $e_i$ and $e_j$, that are contained in
the $0$-edges $E$ and
$E'$ respectively. Relabeling if necessary, we can assume that
$e_i=e_1\subset E$.

Now, it is entirely possible that $e_j= e_2$ or that
$e_j=e_N$. In this case there are no $n$-edges $e_k$ incident to
$p$ contained in the sector between $E$ and $E'$, respectively
in the sector between $E'$ and $E$. We then say that the \emph{sectors}
between $E$ and $E'$
are \emph{not subdivided} by the $n$-edges,
otherwise we say that the \emph{sectors} between $E$ and $E'$
\emph{are subdivided} by the $n$-edges.
It should be pointed out, that is possible that $p$ is of
Fatou-type and the sectors between $E$ and $E'$ are not
subdivided by $n$-edges for any $n\in \N$.

The following theorem however means that we can avoid this
phenomenon if desired.

\begin{theorem}
  \label{thm:invC_subdivSector}
  Let $f\colon S^2\to S^2$ be an expanding Thurston map. Then
  for each sufficiently high iterate $F=f^n$ there is an
  $F$-invariant Jordan curve $\CC\subset S^2$ with
  $\post(f)\subset \CC$, that has the additional property that
  for any Fatou-type postcritical point $p$ the sectors between
  the $0$-edges incident to $p$ are subdivided by the $n$-edges.
\end{theorem}

Here ``$n$-edges'' are $n$-edges of $(f,\CC)$, i.e., $1$-edges
for $(F,\CC)$.  To prove the theorem above, we need to find a
Jordan curve $\CC'\subset f^{-n}(\CC)$ satisfying the conditions
from Theorem~\ref{thm:invC_fromCC1} with the following
additional property. At each Fatou-type postcritical point $p$,
the sectors between the $0$-edges $E,E'\subset \CC'$ incident to
$p$ are subdivided by the $n$-edges. Equivalently, the $n$-edges
$e$ and $e'$ in $\CC'$ that are incident to $p$ have the
property that in each of the sectors between them, there is at
least one $n$-edge incident to $p$.

Let $\CC\subset S^2$ be a Jordan curve with $\post(f) \subset
\CC$, which we orient arbitrarily. Let $\delta>0$ and $\epsilon_0$ be the constants from
Proposition~\ref{prop:isotreln} for $P= \post(f)$. Let $n\in \N$
be sufficiently large such that the diameter of each $n$-tile
$X$ is smaller than $\epsilon_0$ and $\delta$,
\begin{equation}
  \label{eq:n_suff_large}
  \diam(X)< \min\{\epsilon_0,\delta\}.
\end{equation}
Furthermore, we assume that
$n\in \N$ is
sufficiently large, so that according to
Theorem~\ref{thm:ex_C1_in_C} there is a Jordan curve $\CC'\subset
f^{-n}(\CC)$ that traverses the postcritical points in the same
cyclic order as $\CC$ and for any two postcritical points $p_i,
p_{i+1}$ that are consecutive on $\CC$ we have $\CC'(p_i,
p_{i+1}) \subset \mathcal{N}^\delta(\CC(p_i, p_{i+1}))$.

From Theorem~\ref{thm:invC_fromCC1} it follows that there
is an $f^n$-invariant Jordan curve $\widetilde{\CC}\subset S^2$
with $\post(f) \subset \widetilde{\CC}$.

Now consider a Fatou-type postcritical point $p$.
Clearly $\CC'$ contains two
$n$-edges $e,e'$ incident to $p$. In the case that there are
other $n$-edges incident to $p$ in each of the two sectors
between $e$ and $e'$ it follows from the second statement in
Theorem~\ref{thm:invC_fromCC1} that the $f^n$-invariant curve
$\widetilde{\CC}$ has the property that the sectors
at $p$ are subdivided.

If the sectors between $e,e'$ are not subdivided by $n$-edges we
will adjust $\CC'$ in the following. We will ensure that we can still apply
Proposition~\ref{prop:isotreln}, and that the $f^n$-invariant
curve resulting from Theorem~\ref{thm:invC_fromCC1} has the
property that the $n$-edges subdivide the sectors at $p$.

Let us consider the union of all $n$-tiles containing $p$, i.e.,
\begin{equation*}
  V^n(p):= \bigcup\{X\in \X^n \mid p\in X\}.
\end{equation*}
Clearly, this is a closed set whose boundary is a union of
$n$-edges. In the terminology of \cite{bonk:meyer:expanding},
this is the closure of the $n$-flower of $p$.

Let $q$ and $q'$
be the postcritical points that precede and succeed
$p$ on $\CC'$. Let $a$ be the first point on $\CC'(q,p)$
that intersects $V^n(p)$, and $b$ be the last point on
$\CC'(p,q')$ that intersects $V^n(p)$. Clearly $a$ and $b$
are $n$-vertices.

\begin{lemma}
  \label{lem:replaceCab}
  In the setting as above assume that $\gamma(a,b)$ is an
  $n$-edge path contained in $V^n(p)$ with $p\in \gamma(a,b)$
  that satisfies
  \begin{equation*}
    \gamma(a,b)\cap \CC'(q,a) = \{a\}
    \text{ and }
    \gamma(a,b)\cap \CC'(b,q') = \{b\}.
  \end{equation*}
  Then replacing $\CC'(q, q')$ by $\CC'(q,a) \cup
  \gamma(a,b) \cup \CC'(b,q')$ results in a Jordan curve
  $\CC''\subset f^{-n}(\CC)$ that still satisfies the
  assumptions of Proposition~\ref{prop:isotreln}.
\end{lemma}

This means that $\CC''$ traverses $\post(f)$ in the same
cyclical order as $\CC$ and for postcritical points
$p_i,p_{i+1}$ that are consecutive on $\CC$ we have
$\CC''(p_i,p_{i+1}) \subset \mathcal{N}^\delta(\CC(p_i,
p_{i+1}))$.

\begin{proof}
  We first note that $\gamma(a,b)$ does not intersect any side
  $\CC'(p_i,p_{i+1})$ that is distinct from $\CC'(q,p)$ and
  $\CC'(p,q')$. If this were the case, then any point in this
  intersection would be contained in an $n$-tile
  $X\subset V^n(p)$. Since $p\in X$, it follows that $X$ joins
  opposite sides of $\CC'$, which is a contradiction.

  Since $\gamma(a,b)$ is in the $\delta$-neighborhood of $p\in
  \CC$, it follows that $\CC''(q,p) \subset
  \mathcal{N}^\delta(\CC(q, p)$ and $\CC''(p,q')\subset
  \mathcal{N}^\delta(\CC(p, q'))$. This finishes the proof. 
\end{proof}

\begin{figure}
  \centering
  \begin{overpic}
    [width=8cm, tics=20,
    ]{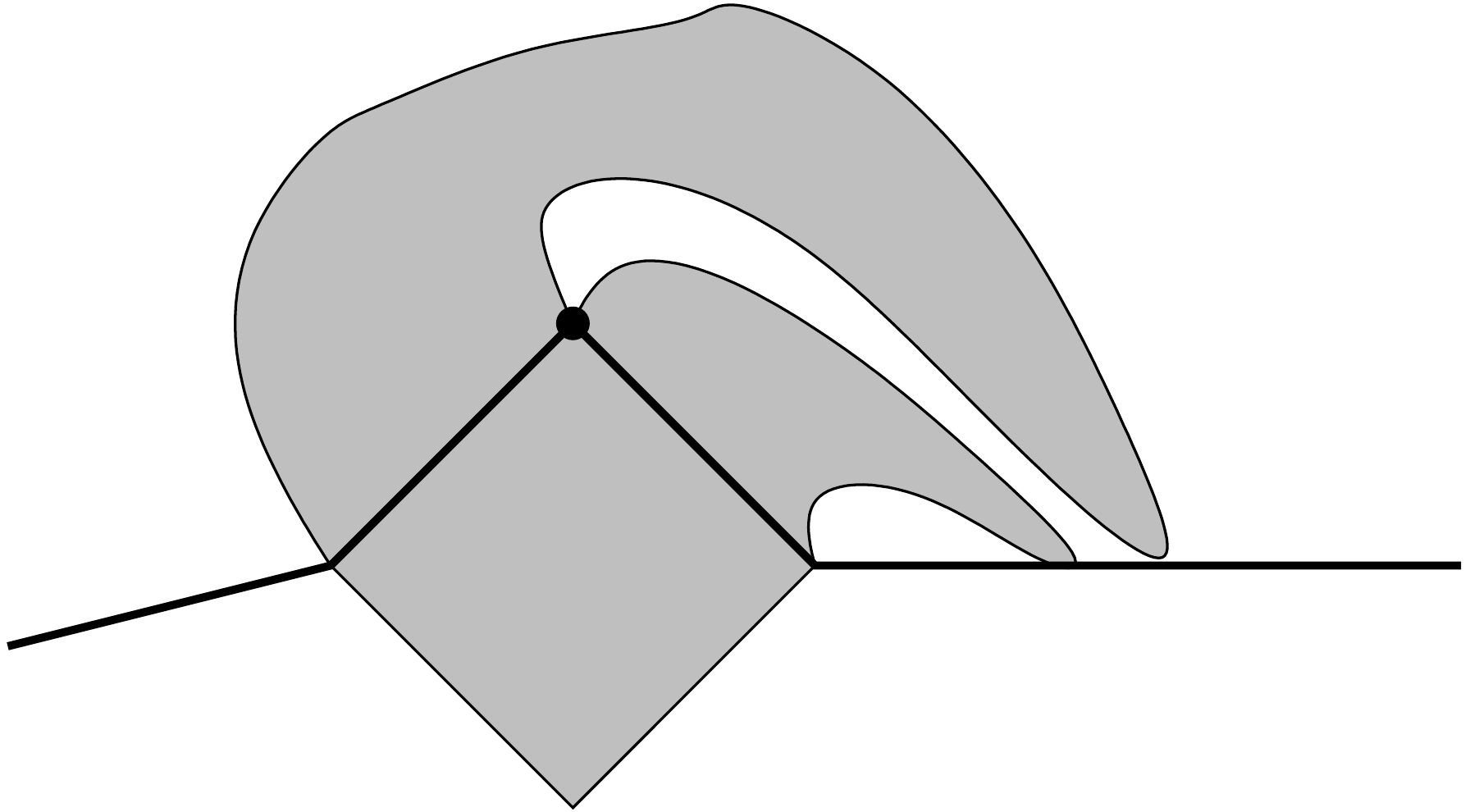}
    \put(36,16){$X$}
    \put(37.5,28){$p$}
    \put(27,38){$Y$}
    \put(52.5,26){$Y'$}
    \put(29,26){$e$}
    \put(44.5,28){$e'$}
    \put(-3,10.5){$q$}
    \put(100.5,16){$q'$}
    \put(78,13){$z$}
    \put(70,12){$y'$}
  \end{overpic}
  \caption{Constructing an $n$-edge path that subdivide
    sectors.}
  \label{fig:divide_sectors}
\end{figure}
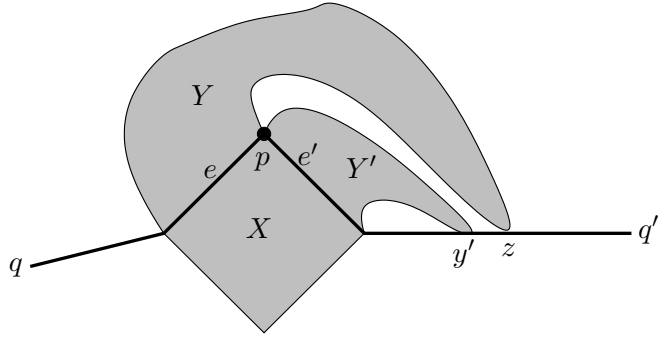

\begin{lemma}
  \label{lem:replace_gab}
  Assume there are at least $4$ $n$-edges incident to $p$. Then
  there is an $n$-edge path $\gamma(a,b)$ that satisfies the
  assumptions of Lemma~\ref{lem:replaceCab} and 
  additionally that the two $n$-edges
  $e,e'\subset \gamma(a,b)$ incident to $p$ have the property
  that in each of the two sectors between them there are other
  $n$-edges incident to $p$.
\end{lemma}

\begin{proof}
  The proof is illustrated in Figure~\ref{fig:divide_sectors}.
  Let $e\subset \CC'$ and $e'\subset \CC'$ be the $n$-edges in
  $\CC'$ that precede and
  succeed $p$ in $\CC'$. If
  each of the two sectors between $e$ and $e'$ contains at least
  one $n$-edge incident to $p$, there is nothing to prove.

  Assume now that one sector between $e$ and $e'$ does not
  contain any $n$-edge incident to $p$. Then $e$ and $e'$ are
  contained in a common $n$-tile $X$ incident to $p$, i.e.,
  $X\subset V^n(p)$. Let $Y\subset V^n(p)$ be the $n$-tile
  distinct from $X$ that contains $e$, and $Y'\subset V^n(p)$ be
  the $n$-tile distinct from $X$ that contains $e'$. Since we
  assume that there are at least $4$ $n$-edges incident to $p$,
  it follows that $Y$ and $Y'$ are distinct.

  Let $y\in \partial Y$ be the first $n$-vertex in
  $\CC'(q,p)$ intersecting $Y$. Let
  $\gamma\subset \partial Y$ be the
  $n$-edge path in $Y$ between $y$ and $p$ that does not include
  $e$. Denote by $e_Y$ the last $n$-edge in $\gamma$, i.e., the
  $n$-edge in $\gamma$ incident to $p$.

  Note that one sector between $e'$ and $e_Y$ contains $e$,
  the other sector contains all other $n$-edges incident to $p$.

  If $\gamma$ does not intersect $\CC'(p, q')$ we may
  replace $\CC'(y,p)$ by $\gamma$, finishing the desired
  construction.

  \smallskip
  Let us now assume that $\gamma$ does intersect
  $\CC'(p,q')$. The situation is
  illustrated in Figure~\ref{fig:divide_sectors}. Let
  $z\in \partial Y$ be the first $n$-vertex on $\CC'(p, q')$
  in this intersection. Then $\CC'(p, z)$ and the $n$-edge path
  on $\partial Y$ between $p$ and $z$ that does not contain
  $e$ form a Jordan curve that encloses $Y'$. In particular $Y'$
  intersects $\CC'(q,p)$ only in $p$.

  Now let
  $y'\in \partial Y'$ be the last $n$-vertex on $\CC'(p, q')$
  intersecting $Y'$. We now replace $\CC'(p, y')$ by the $n$-edge
  path $\gamma'$ on $\partial Y'$ between $p$ and $y'$ that
  does not contain $e'$. It follows that $\gamma'$
  intersects $\CC'(q,p)$ only in $p$. Let $e_{Y'}$ be the
  $n$-edge in $\gamma'$ incident to $p$.
  One sector between $e$ and $e_{Y'}$ contains $e'$. The other
  sector 
  contains the other $n$-edges incident to $p$. This finishes
  the proof. 
\end{proof}

\begin{proof}[Proof of Theorem~\ref{thm:invC_subdivSector}]
  Let $n\in \N$ be sufficiently large such that the assumptions
  stated in the beginning of this section are
  satisfied. Furthermore, we assume that any Fatou-type
  postcritical point is incident to at least $4$ $n$-edges.

  Then from the Jordan curve $\CC'$ we can construct the Jordan
  curve $\CC''$ by changing the curve in $V^n(p)$ for each
  Fatou-type postcritical point $p$ as in
  Lemma~\ref{lem:replace_gab}. Theorem~\ref{thm:invC_fromCC1}
  finishes the proof.
\end{proof}

Let $\CC\subset S^2$ be an $f^n$-invariant Jordan curve with
$\post(f)\subset \CC$ such that the sectors between $0$-edges
are subdivided by $n$-edges for some $n\in \N$ as given by
Theorem~\ref{thm:invC_subdivSector}. Note that for any iterate
$f^{nk}$ of $f^n$ the Jordan curve $\CC$ is again
$f^{nk}$-invariant and the sectors at Fatou-type postcritical
points are again subdivided by $nk$-edges.

\section{Expanding quotients}
\label{sec:expanding-quotients}

In this and the following section we prove
Theorem~\ref{thm:main}. Let $R\colon \cbar\to \cbar$ be a
postcritically finite rational map with Julia set homeomorphic
to the Sierpi\'nski carpet. Note that $R$ is not expanding in
the sense of Definition~\ref{def:f}. To be able to use
Theorem~\ref{thm:invC_subdivSector}, we take a quotient of $R$
that will turn out to be an expanding Thurston map.

Define an equivalence relation on $\cbar$ by $z\sim w$ if
$z=w$ or if there is a Fatou component $\Omega$ such that $z,w\in
\overline{\Omega}$. Since $\mathcal{J}_R$ is a Sierpi\'nski
carpet it follows that distinct Fatou components $\Omega$ and
$\Omega'$ have disjoint closures. So $\sim$ is indeed an
equivalence relation.
Let us denote by $\pi\colon \cbar  \to
\cbar/\!\sim$ the quotient map.
Clearly $\pi$ is \emph{monotone}, which means that for every
point $x\in \cbar/\!\sim$ the set $\pi^{-1}(x)\subset \cbar$ is connected,
since each equivalence class is connected. This
implies that for any connected set $A\subset \cbar/\!\sim$ the
set $\pi^{-1}(A) \subset \cbar$ is connected, see
\cite[VIII.2.2]{whyburn:analytic_topology}.

\begin{theorem}
  \label{thm:quotient}
  In the setting as above we may identify $\cbar/\!\sim$ with
  $S^2$ so that the quotient map may be written as the
  continuous monotone map $\pi \colon
  \cbar \to S^2$. Furthermore there exists an expanding Thurston
  map
  $f:S^2\to S^2$ such that $f\circ \pi = \pi\circ R$.
\end{theorem}

We rely on the following notion and fact.
An equivalence relation $\sim$ on the sphere $S^2$ is said to
be of \emph{Moore-type}, if
\begin{enumerate}
\item $\sim$ is not trivial, meaning that there are at least two
  distinct equivalence classes;
\item $\sim$ is closed (here we view the
relation $\sim$ as a subset of $S^2\times S^2$ equipped with
the product topology);
\item each equivalence class is a compact
connected set;
\item no equivalence class separates $S^2$. This means
  that for any equivalence class $[x]$ of $\sim$ the set $S^2
  \setminus [x]$ is connected.
\end{enumerate}
Moore's theorem is the fact that the quotient
$\widetilde{S}^2:= S^2/\!\!\sim $ endowed with
the quotient topology is homeomorphic to $S^2$ \cite[Theorem
IX.2.1']{whyburn:analytic_topology}.

Now let $f\colon S^2\to S^2$ be a Thurston map.
We call $\sim$ \emph{strongly $f$-invariant} if the image of any
equivalence class is an equivalence class, i.e.,
\begin{equation}
  \label{eq:def_strong_inv}
  f([x]) = [f(x)],
\end{equation}
for any $x\in S^2$. The purpose of this condition is explained
by the following, see
\cite[Corollary~13.8]{bonk:meyer:expanding} for a proof.

\begin{lemma}
  \label{lem:quotient_thurs}
 Let $f\colon S^2 \to S^2$ be a Thurston map and $\sim$ be a
  strongly $f$-invariant
  equivalence relation of Moore-type on $S^2$. Then there is a
  Thurston map $\widetilde{f}\colon \widetilde{S}^2\to
  \widetilde{S}^2$ such that
  \begin{equation*}
    \pi\circ f= \widetilde{f}\circ \pi.
  \end{equation*}
\end{lemma}


The existence of $f$ follows essentially from the above in
conjunction with Moore's theorem.

\begin{lemma}
  \label{lem:sim_Moore_str_inv}
  There is a Thurston map $f\colon S^2\to S^2$ such that $f\circ
  \pi = \pi\circ R$.
\end{lemma}

\begin{proof}
  According to Lemma~\ref{lem:quotient_thurs}.
  we
  need to show that $\sim$ is of Moore-type and strongly
  $R$-invariant.

  Clearly $\sim$ is not trivial. Furthermore each equivalence
  class is either a single point or a closed Jordan domain,
  thus compact, connected, and it does not separate $\cbar$.

  To prove that $\sim$ is closed it suffices to show that given
  two convergent sequences $(x_n)_{n\in\N}$ and $(y_n)_{n\in
    \N}$ in $\cbar$ with $x_n\sim y_n$ for all $n\in \N$ it
  follows that $\lim x_n \sim \lim y_n$. This is clear in the
  case when for sufficiently large $n\in \N$ the points $x_n$ are
  contained in some fixed equivalence class, since each
  equivalence
  class is compact. Otherwise, we may assume that for distinct
  $n,m\in \N$ the points $x_n$ and $x_m$ are contained in
  distinct equivalence classes. In this case
  Lemma~\ref{lem:comp_F_small} shows that $\lim x_n = \lim
  y_n$. Thus $\sim$ is closed.
  We conclude that $\sim$ is of Moore-type. Thus $\cbar/\sim$ is
  homeomorphic to $S^2$, so we may identify the two spaces.

  Note that if $x\in \cbar$ is not contained in the closure of
  a component of the Fatou
  set, the same it true for $R(x)$. Furthermore $R$ maps the
  closure of each component of the Fatou set to the closure of a
  component of the Fatou set. Thus $\sim$ is strongly
  $R$-invariant. Hence $f\colon S^2\to S^2$ is a Thurston map
  that satisfies $f\circ \pi = \pi \circ R$ by
  Lemma~\ref{lem:quotient_thurs}.
\end{proof}

So it remains to show that $f$ is expanding. We will need some
preparation.



The set $F=\pi(\mathcal{F}_R)$ is countable.
Let us say that an open set $U$ in $\cbar$ is {\it clean} if
$\partial U \cap \overline{\Omega}=\emptyset$ for all Fatou
components $\Omega$. If $U$ is a clean neighborhood it
is saturated with respect to $\sim$ and the
restriction $\pi|_{\partial U}$ is a
homeomorphisms onto its image.

\begin{lemma}\label{lma:clean} Let $\xi$ be an equivalence class. For any open subset $U$ containing $\xi$, there exists a clean
open set $V$ with boundary a simple closed curve such that $\xi\subset V\subset U$.\end{lemma}

\begin{proof}
Identify  $\pi(\cbar)=S^2$ with the Euclidean  sphere.  Since $S_r=\{x\in S^2,\ d(\pi(\xi),x)= r\}$, $r>0$, is an uncountable family
of pairwise disjoint sets and $F(=\pi(\mathcal{F}_R))$ is
countable, we may find $(r_n)$ tending to $0$ such that
$S_{r_n}\cap F =\emptyset$
The set $V_n= \pi^{-1}(\{x\in S^2,\ d(\pi(\xi),x)<r_n\})$ is open and clean by construction. Furthermore,
$$\cap V_n = \pi^{-1}(\cap   \{x\in S^2,\ d(\pi(\xi),x)<r_n\}   )=\xi$$ holds
so that for $n$ large enough, one has $\xi\subset V_n\subset U$.
\end{proof}




\begin{lemma}\label{lma:shrinknbhd}
For any $x\in S^2$
there is a neighborhood $V\subset S^2$ of $x$ such that the maximal
diameter of any component of $f^{-n}(V)$ tends to $0$ as $n\to \infty$.
\end{lemma}

\begin{proof}
By the shrinking lemma, one may find an open simply connected
neighborhood  $U_z\subset \cbar$ for each $z\in \mathcal{J}_R$
such that
the  components of $R^{-n}(U_z)$ are as small as wanted provided $n$ is large enough. Since $R$ is postcritically finite,
either $z\notin\post(R)$ and so we may choose $U_z$ disjoint from
$\post(R)$ or $z\in\post(R)$ and so we may assume $U_z\cap
\post(R)=\{z\}$.

Let $x\in S^2$ be arbitrary and set $\xi=\pi^{-1}(\{x\})$ so that $\xi\subset \cbar$ is an equivalence class.
Let $U_\xi'$ be the union of $\xi$ and the union of all $U_z$'s for $z\in \xi\cap \mathcal{J}_R$
(note that if $\xi=\{z\}$ is a singleton then $U_\xi'=U_z$).
According to Lemma \ref{lma:clean}, we may find a clean and simply connected open set $U_\xi$
such that $\xi\subset U_\xi\subset U_\xi'$.
Since $\mathcal{J}_R$ is a carpet, $U_\xi$ contains at most one point from $\post(R)$ which is contained
in $\xi$, see Section~\ref{sec:sierp-carp-julia}.

Endow $S^2$ with a metric compatible with its topology and let $\ep>0$; the uniform continuity of $\pi$
implies the existence of $\de>0$ such that sets of diameter less than $\de$ are mapped under $\pi$ to sets
of diameter at most $\ep/2$. By choosing $n$ large enough, we
can ensure that each of the components $W$  of $R^{-n}(U_\xi)$ will be covered by the union
of a single preimage $\zeta$ of $\xi$ and of components of
$R^{-n}(U_z)$, $z\in\xi$, each of which intersects $\zeta$ and has diameter at most $\de$.
Thus $\diam\, \pi(W)\le \ep$.

Now define $V:= \pi(U_\xi)\subset S^2$. Since $U_\xi$ is
saturated and
open, it follows that $V$ is an open neighborhood of $x\in
S^2$. Let $V'\subset S^2$ be a component of $f^{-n}(V)$. Since
$\pi$ is monotone, it follows that $\pi^{-1}(V')\subset \cbar$ is
connected. This implies in conjunction with the identity $f\circ \pi = \pi
\circ R$ that $\pi^{-1}(V')$ is contained in a
component of $R^{-n}(U_\xi)$. Thus $\diam (V')\leq \epsilon$ as
desired.
\end{proof}

\begin{proof}[Proof of Theorem \ref{thm:quotient}]
According to \cite[Proposition~6.3]{bonk:meyer:expanding},
the map $f$ is an expanding Thurston map if there is a finite covering $\VV$ of $S^2$ by connected open sets
such that $(\VV_n)_n$ defines
a basis for the topology of $S^2$ where
$\VV_n$ denotes the set of connected components of $f^{-n}(V)$ for  $V\in\VV$.
 Apply  Lemma \ref{lma:shrinknbhd} and let $\VV$ be a finite subcover. The theorem follows.\end{proof}

\section{Lifts of Jordan curves}
\label{sec:lifts-jordan-curve}

We are now ready to prove Theorem~\ref{thm:main}.
Let
$R\colon \cbar \to \cbar$ be a postcritically finite rational
map with Sierpi\'{n}ski carpet Julia set. From Theorem
\ref{thm:quotient} we obtain the expanding Thurston map $f\colon
S^2\to S^2$ and the monotone map $\pi\colon \cbar \to
S^2$ such that $f\circ \pi = \pi \circ R$. As in the last
section we let $F:= \pi(\mathcal{F}_R)$.

The idea to construct the $R^n$-invariant Jordan curve
$\Gamma\subset \cbar$ is to start with an $f^n$-invariant Jordan
curve $\CC\subset S^2$ given by
Theorem~\ref{thm:invC_subdivSector} and extract an
$R^n$-invariant Jordan curve in $\pi^{-1}(\CC)$. By the remark at the end of Section~\ref{sec:invar-jord-curv}, we
may take further iterates $f^{nk}$ (and correspondingly
$R^{nk}$) if convenient. To simplify the
discussion, we assume in the following however that $n=k=1$.

\begin{theorem}
  \label{thm:lifts}
  With the above notation, let $\CC\subset S^2$ be an
  $f$-invariant Jordan with $\post(f) \subset \CC$, such that at
  each Fatou-type
  postcritical point the sectors between $0$-edges are
  subdivided by $1$-edges.
  Then there exists an $R$-invariant Jordan curve $\G\subset
  \cbar$ satisfying $\post(R) \subset \Gamma$ as well as
  $\pi(\Gamma)= \CC$.
\end{theorem}

Note that if $\CC\subset S^2$ is any Jordan curve, there is in
general no Jordan curve $\Gamma\subset \cbar$ with
$\pi(\Gamma)= \CC$.

\begin{ex}
  \label{ex:no_Jordan_in_lift}
  Consider the closed unit disk $\overline{\D}\subset \C$. Let
  $\gamma, \gamma'\subset \C\setminus \overline{\D}$ be two
  disjoint arcs with the following properties. The arc $\gamma$
  starts at the point $-2\in \R\subset \C$ and winds
  infinitely often around $\overline{\D}$, so that $\gamma$
  accumulates on $\partial \D$. Similarly $\gamma'$ starts at
  $2\in \R\subset \C$ and winds infinitely often around
  $\overline{\D}$ so that it accumulates on $\partial \D$.

  Collapse the set $\overline{\D}$ to a point. More precisely, we
  consider the equivalence relation on $\C$ where
  $\overline{\D}$ is an equivalence class, all other equivalence
  classes are of the form $\{z\}$ where $z\in \C\setminus
  \overline{\D}$. Let $\pi\colon \C\to \C/\!\sim$ be the
  quotient map. It is easy to see that $\pi(\gamma\cup
  \overline{\D} \cup \gamma')$ is an arc, yet there is no Jordan
  arc contained in $\gamma\cup \overline{\D}\cup \gamma'$
  connecting $-2$ and $2$.
\end{ex}

We start with several lemmas.

\begin{lemma}
  \label{lem:pre_Jordan}
  Let $\CC\subset S^2$ be a Jordan curve. Then $\cbar\setminus
  \pi^{-1}(\CC)$ has exactly two components.
\end{lemma}

\begin{proof}
Let $U_1$ and $U_2$ be the two components of $S^2\setminus
\CC$.
Since $\pi$ is monotone, the sets $V_1:= \pi^{-1}(U_1)\subset
\cbar$ and $V_2:= \pi^{-1}(U_2)\subset \cbar$ are connected. If
$V_1\cup V_2$ would be connected, it would follow that $U_1
\cup U_2 =\pi(V_1 \cup V_2)$ is connected, which is a
contradiction.
\end{proof}

To avoid problems as in Example~\ref{ex:no_Jordan_in_lift}, we
will use that $\CC$ is $f$-invariant and the
shrinking lemma of the orbifold metric for $R$ in an essential
way.

In the next lemma we consider a Jordan arc $c\colon [0,1] \to
S^2$ 
starting at $x= c(0)$ that is invariant for $f$
in the following sense. We have $f(x)= x$ and there is an $s\in
(0,1)$ such that $f(c([0,s]))= c([0,1])$.

\begin{lemma}
  \label{lma:invgerm}
  Let $c:[0,1]\to S^2$ be a Jordan arc that is invariant for $f$
  as 
  above and assume that $c(0,1]\cap \post(f)=\emty$.
  Then $\pi^{-1}(c(0,1])$ has a single accumulation
  point which is a fixed point for $R$.

  Furthermore if $c':[0,1]\to S^2$ is a preimage of $c$ by some
  other iterate $f^\ell$, i.e., $c'$ is another Jordan arc such
  that 
  $f^\ell(c') =c$, then $\pi^{-1}(c'(0,1])\subset \cbar$ has a
  single accumulation point as well.
\end{lemma}

\begin{proof}


Using Lemma~\ref{lma:shrinknbhd} and considering an iterate of
$f$ if
necessary, we may find a neighborhood $V\subset S^2$ with the
properties that $x$ is the only
possible postcritical point in $V$ and
that $W=f^{-1}(V)\cap V$ is compactly contained in $V$. Pick $y = c(t) \in W\setminus F$ with $t>0$, such that $c[0,t]\subset W$
and denote by $c_0$ the closure of the connected component of $c\setminus \{y,f(y)\}$ bounded by $\{y,f(y)\}$. Define inductively
$(c_n)_{n\ge 0}$ so that $c_{n+1}$ is the lift of $c_n$ by $f$, where $c_n\cap c_{n+1}$ is the singleton $\{y_n\}$ such that
$f^{n}(y_n)=y$.
It follows that $\cup_{n\ge 1} c_n = c(0,t]$.
Let $\kappa_n=\pi^{-1}(c_n)$, for each $n\ge 0$.  Since $\pi$ is monotone, $\kappa_n$ is connected and $\kappa_{n+1}\cap\kappa_n$
is a singleton which is eventually mapped by $R^n$ to
$\pi^{-1}(y)$. Note that $\cup_n\kappa_n=  \pi^{-1}(c(0,t])$ is
connected, again by the monotonicity of $\pi$. 

It follows from \cite[Lemma 17.3]{daverman:book} that $\kappa_0$ is contractible, so that we may include it into a simply connected open
set $D_0\subset W$ avoiding $\post(R)$. For each $n\ge 1$, we may define an inverse branch $R_{-n}:D_0\to \mathcal{O}$ so that $R_{-n}(\kappa_0)=\kappa_n$.
Since $D_0$ intersects $\mathcal{J}_R$, the sequence $\{R_{-n}\}_n$  is a normal family and any limit is constant. Therefore, we may find
$n_0$ such that $\diam_{\mathcal{O}} \kappa_{n_0}\le \de$. 
By the shrinking lemma,
we have $\diam_{\OO} \kappa_{n}\lesssim \diam_{\OO}\kappa_0/\la^n$ for all $n\ge 1$. This implies that
$\diam_\OO\cup_{n\ge k}\kappa_n \lesssim \la^{-k}$. Hence  $(\kappa_n)_n$
converges to a single point $z\in \pi^{-1}(\{x\})$. Since
$R(\kappa_n) = \kappa_{n-1}$ converges to $z$ as well, it
follows that $z$ is a fixed point of $R$.

\smallskip
To see the second statement we note that the set
$\pi^{-1}(c'(0,1])\subset \cbar$  is connected and its cluster
set is connected as well as contained in
$\pi^{-1}(c'(0))$. Since $R$ is open and
\begin{equation*}
  R^\ell(\overline{\pi^{-1}(c'(0,1])})
  \subset
  \overline{\pi^{-1}(c(0,1])},
\end{equation*}
it follows that $\pi^{-1}(c'(0,1])$  has a single accumulation
point.
\end{proof}

\begin{corollary}
  \label{cor:xinF_lifts}
  Let $\CC\subset S^2$ be an $f$-invariant Jordan curve, $x\in F\cap\CC$, and
  $c\colon [0,1] \to \CC$ be a Jordan arc based at
  $x$, meaning that $c(0)=x$. Then $\pi^{-1}(c(0,1])\subset
  \cbar$ accumulates at a single point.
\end{corollary}

\begin{proof}
  Since the point $x$ is preperiodic, there are $n\geq 0$ and
  $k\geq 1$ such that the arc $c_n:=f^n\circ
  c$ is $f^k$-invariant. From Lemma~\ref{lma:invgerm}  it follows
  that $\pi^{-1}(c_n((0,1]))\subset \cbar$ has a single
  accumulation point. 
  From the second part in this lemma it follows that
  $\pi^{-1}(c((0,1]))\subset \cbar$ has a single accumulation
  point as desired.
\end{proof}

Now consider the two arcs on
$\CC$ incident to some point $x\in F\cap \CC$. By the above,
their preimages by
$\pi$ accumulate both at a single point in the boundary of
$\pi^{-1}(x)$, which is the closure of a Fatou component. In
order to construct the desired Jordan curve 
$\Gamma$, these two accumulation points need to be
different. This is in fact the case when at each
Fatou-type postcritical $p$ the sectors between the $0$-edges
incident to $p$ are subdivided by the $1$-edges, see
Section~\ref{sec:invar-jord-curv} for the terminology. This is
the reason we introduced this notion and proved
Theorem~\ref{thm:invC_subdivSector}.

\begin{lemma}
  \label{lma:liftg}
  Let $\CC\subset S^2$ be an $f$-invariant Jordan curve with $\post(f) \subset
  \CC$ such that the
  sectors between $0$-edges of every Fatou-type postcritical
  point is subdivided by $1$-edges. Then for any $x\in F\cap
  \CC$ the set $\pi^{-1}(\CC\setminus
  \{x\})$ accumulates in $\pi^{-1}(x)$ in two distinct points.
\end{lemma}

\begin{proof}
  Assume first that $x=p\in S^2$ is
  a periodic Fatou-type postcritical point of $f$. Let
  $\overline{\Omega}_p:= \pi^{-1}(p)\subset \cbar$. It is the
  closure of a Fatou component $\Omega_p$ of $R$ that contains a
  Fatou-type postcritical point of $R$.

  Let $e,e'\subset \CC$ be the $1$-edges incident to $p$. If
  $K:=\pi^{-1}(e\setminus \{p\})$ and $K':=\pi^{-1}(e'\setminus
  \{p\})$
  accumulate at distinct points in $\partial \Omega_p$ we
  are done. So let us assume from now on that these sets both
  accumulate at the same point $z\in \partial \Omega_p$.

By assumption there is a
$1$-edge $e_0$ between $e$ and $e'$ and a $1$-edge $e_1$ between
$e'$ and $e$. Let $\widehat{e}_0$ and $\widehat{e}_1$ be the
interiors of these $1$-edges, i.e., we have removed the two
endpoints from $e_0$ as well as $e_1$. Then $\widehat{e}_0$ and
$\widehat{e}_1$ are disjoint.
Let $z_0$ and $z_1$ be
the accumulation points of $K_0:=\pi^{-1}(\widehat{e}_0)$ and
$K_1:= \pi^{-1}(\widehat{e}_1)$ in $\partial{\Omega_p}$
respectively.  Lemma~\ref{lem:pre_Jordan} implies that $K_0$ and
$K_1$ are in distinct components of $\pi^{-1}(S^2\setminus\CC)$.
Note that
since $K$ and $K'$ accumulate at the same point of
$z\in \overline{\Omega}_p$, it follows that the closure of
$\pi^{-1}(\CC\setminus\{x\})$ separates $\Omega_p$ from $K_0$ or
$K_1$. Without loss of generality, we may assume that it is
$K_0$. This means that $z=z_0$.

Note that $f(e), f(e'), f(e_0)$ are $0$-edges incident to
$q=f(p)$. Since there are only two distinct such $0$-edges, it
follows that at least  two of them have to coincide;  let us assume $f(e)=f(e')$.
Hence, $\pi^{-1}(f(e)\setminus \{q\})=\pi^{-1} (f(e')\setminus\{q\})$ holds as well.
But since  $f\circ \pi =  \pi\circ R$, it follows that $R(K)=R(K')$. But this
contradicts the fact that $R$ is one-to-one in a neighborhood of $z$.
This contradiction shows that  $z\neq z'$ as desired.

  \smallskip
  Now let $x\in F$ be arbitrary. Then there is an $k\in \N$ such
  that $p:=f^k(x)$ is a periodic postcritical point. The two
  $(k+1)$-edges $e,e'\subset \CC$ incident to $x$ are mapped to
  (possibly the same)
  $1$-edges in $\CC$ incident to $p$. Thus, the $(k+1)$-edges
  incident to
  $x$ divide the sectors between $e$ and $e'$. The argument now
  proceeds exactly as above.

\end{proof}

We are now ready to prove Theorem~\ref{thm:lifts}.

\begin{proof}[Proof of Theorem~\ref{thm:lifts}]
Let $\CC\subset S^2$ be an $f$-invariant Jordan curve with
$\post(f) \subset \CC$, such that
at each Fatou-type postcritical point the sectors between
$0$-edges are subdivided by the $1$-edges as in
Section~\ref{sec:invar-jord-curv}.
Let $K:= \pi^{-1}(\CC)\subset \cbar$ and $L\subset \cbar$ be the
closure of $\pi^{-1}(\CC\setminus F)$. Note that for any Fatou
component $\Omega\subset K$, the set $L\cap \overline{\Omega}=
L\cap \partial \Omega$ consists of exactly two distinct points
by Lemma~\ref{lma:liftg}.

We let $\G\subset\cbar$ be the union of $L$ and the set of
internal rays which join a point of $L\cap\partial\Omega$ to
the center of $\Omega$ for each Fatou component $\Omega\subset
K$. Since the Fatou set $\mathcal{F}_R$ is invariant for $R$, it
follows that $F\cap \CC$ is forward invariant by $f$. From
$f\circ \pi = \pi\circ R$ it
follows that $L$ is forward invariant for $R$. Hence any
internal ray
in $\Gamma$ is mapped to another internal ray in $\Gamma$. Thus
$\Gamma$ is (forward) $R$-invariant, meaning that 
$R(\Gamma)\subset \Gamma$.

Note that $\pi(\post(R)) = \post(f)$ (see
\cite[Theorem~13.6]{bonk:meyer:expanding}). Thus $\post(R)
\subset \pi^{-1}(\post(f))$.
So any $q\in
\post(R)$ is contained in the set $\pi^{-1}(p)$ for some $p\in
\post(f)$. Note that $\pi^{-1}(p)$
 is either $\overline{\Omega}_q
\subset K$, the closure of the Fatou
component of $R$ whose center is $q$, or $\pi^{-1}(p) =
\{q\}\subset L$. In
any case $q\in \Gamma$ by construction, meaning that $\post(R)
\subset \Gamma$. Finally, the fact that $\pi(\Gamma) = \CC$ is
obvious from the construction as well.

\begin{claim*}
  $\G$ is a Jordan curve.
\end{claim*}
First, we define a retraction map $\psi:K\to \G$ as follows. If
$z\in L$, then set $\psi(z)=z$.
For any Fatou component $\Omega\subset K$ there is monotone
retraction $\rho:\overline{\Omega}\to \overline{\Omega}\cap\G$
mapping the closed Jordan domain to the crosscut defined by
$\G\cap\overline{\Omega}$; we let $\psi=\rho$ there.

We claim that $\psi$ is a retraction. The continuity follows from the fact that $\mathcal{J}_R$ is an $E$-continuum (cf.\ p.\,\pageref{Econt}).
Moreover, $\psi$ is a retraction since it is the case on the
closure of each Fatou component contained in $K$ and is the
identity elsewhere. Therefore, $\Gamma$ is continuum.

According to \cite[Theorem III.7]{whyburn:analytic_topology}, it
is sufficient to prove that any pair of distinct points of $\G$
separates 
$\G$.  Let us  consider two points  $\{z,w\}\subset \G$ and set $x=\pi(z)$ and $y=\pi(w)$.

If $x=y$, then $z,w$ are in the closure of the same Fatou
component $\Omega\subset K$, so it follows easily that
$\G\setminus\{z,w\}$ is disconnected.  Let us now assume that
$x\ne y$. Since $\CC$ is a Jordan curve, $\{x,y\}$ separates
$\CC$, hence, since $\pi$ is continuous,
$\{\pi^{-1}(x),\pi^{-1}(y)\}$ separates $K$ and, since $\psi$ is
monotone, $\{\psi(\pi^{-1}(x)),\psi(\pi^{-1}(y))\}$ separates
$\G$.  Now, $I_x=\psi(\pi^{-1}(x))$ and $I_y=\psi(\pi^{-1}(y))$
are intervals which contain $z$ and $w$ respectively. If these
are non
degenerate, we may define retractions
$\rho_z: I_x\setminus \{z\}\to \partial I _x\setminus\{z\}$ and
$\rho_w:I_y \setminus \{w\}\to \partial I_y\setminus\{w\}$,
yielding a monotone retraction
$\rho:\G\setminus \{z,w\}\to \G\setminus (I_x\cup I_y)$ by
letting $\rho=\id$ on the complement. This enables us to
conclude that
$\{z,w\}$ separates  $\G$. Thus $\Gamma$ is a Jordan curve.

\end{proof}

We are ready to finish the proof of our main theorem.

\begin{proof}[Proof of Theorem \ref{thm:main}]
Let $R\colon \cbar \to \cbar$ be a postcritically finite
rational map whose Julia set is homeomorphic to the
Sierpi\'{n}ski carpet. Let $f\colon S^2\to S^2$ be the expanding
Thurston map as given by Theorem~\ref{thm:quotient}. From
Theorem~\ref{thm:invC_subdivSector} it follows that for each sufficiently high iterate $f^n$ there is an
$f^n$-invariant Jordan curve $\CC\subset S^2$ with
$\post(f)\subset \CC$, that has the additional property that
for any Fatou-type postcritical point $p$ the sectors between
the $0$-edges incident to $p$ are subdivided by the $n$-edges.
Note that $\post(f^n) = \post(f)$ and $\post(R^n) = \post(R)$.
Applying Theorem~\ref{thm:lifts} to $f^n$ and to this curve $\CC$
finishes the proof.
\end{proof}

\begin{rem}
  \label{rem:Jordan_with_hair}
  In the case when $\CC \subset S^2$ with $\post(f) \subset \CC$
  is $f$-invariant, but does
  not necessarily divide sectors at Fatou-type postcritical
  points, the construction given in the proof of
  Theorem~\ref{thm:lifts} still yields an $R$-invariant set
  $\Gamma\subset \cbar$ with $\post(R)\subset \Gamma$. The set
  $\Gamma$ however will not be a Jordan curve in general, but a
  Jordan curve together with countably many internal rays
  connecting the center of Fatou components to this Jordan
  curve. In fact, we may remove all but finitely many of these
  internal rays and still have an $R$-invariant set
  $\Gamma\subset \cbar$ with $\post(R) \subset \Gamma$.
\end{rem}



\section{Acknowledgments}
\label{sec:acknoledgements}

The authors thank Tan Lei, who did initiate this collaboration. 
We are 
grateful to Mario Bonk, Mikhail Hlushchanka, and Kevin Pilgrim for their comments on a preliminary
version.  Y.G and JS.Z thank Professor Cui for helpful suggestions and discussions.
P.H. is partially supported by the  ANR projects ``GDSous/GSG'' no. 12-BS01-0003-01
and ``Lambda'' no. 13-BS01-0002. D.M. has been partially
supported by the Academy of Finland via the Centre of Excellence
in Analysis and Dynamics Research (project No. 271983), as well
as by the Deutsche Forschungsgemeinschaft (DFG-ME 4188/1-1).
P.H and D.M. thank IPAM for their hospitality where part of this work
has been done during the program ``Interaction between Analysis and Geometry''.

\end{document}